\documentclass{amsart}
\usepackage[a4paper,margin=2.2cm]{geometry}
\usepackage{amssymb,esint}
\usepackage{amscd}
\usepackage{xspace}
\usepackage{fancyhdr}
\usepackage{xcolor}
\usepackage{hyperref,amsmath}
\usepackage{cite}
\newtheorem{theorem}{Theorem}[section]
\newtheorem{lemma}[theorem]{Lemma}
\newtheorem{corollary}[theorem]{Corollary}
\theoremstyle{definition}

\theoremstyle{proposition}

\theoremstyle{remark}
\newtheorem{remark}[theorem]{Remark}

\numberwithin{equation}{section}

\begin{document}

\title{Pointwise upper bound for the fundamental solution of fractional Fokker-Planck equation}

\author{Haina Li}
\address{School of Mathematics and Statistics, Beijing Institute of Technology, Beijing 100081, China.}
\email{lihaina2000@163.com}


\author{Yiran Xu}
\address{Fudan University, 220 Handan Road, Yangpu, Shanghai, 200433, China.}
\email{yrxu20@fudan.edu.cn}





\begin{abstract}
	In this paper,  we investigate the fundamental solution of the fractional Fokker-Planck equation. Utilizing the Littlewood-Paley decomposition technology, we present a concise proof of the pointwise estimate for the fundamental solution.
\end{abstract}
\maketitle
\section{Introduction}
We now formally define the class of equations under investigation in this article. 
Specifically, we consider the following linear kinetic fractional Fokker-Planck equation, formulated in the distributional sense:
\begin{equation}\label{aim equation}
	\partial_t f+v\cdot \nabla_x f+|D_v|^{2s} f=0,\quad 0<s<1,
\end{equation}
where $f=f(t,x,v)$ represents the unknown function  dependent on time $t$, position $x$ and velocity $v$, defined on the domain $\mathbb{R}_+\times\mathbb{R}\times\mathbb{R}$. The initial condition is given by $f(t=0)=f_0$.

Before introducing the Fokker-Planck equation, we first review the properties of the heat equation, which serves as a simpler case compared to the Fokker-Planck equation. It is well-established that the classical Laplacian heat equation is given by
\begin{equation*}
	\partial_t u-\Delta u=0,
\end{equation*}	
and its fundamental solution, which originates from a point source (represented by the Dirac delta function), is the Gaussian heat kernal:
\begin{equation*}
	\mathcal{K}(t,x,v)=(4\pi t)^{-\frac{d}{2}} \exp(-\frac{|x-v|^2}{4t}),
\end{equation*}	
where $x,v\in \mathbb{R}^d$ denote spatial points, and $t>0$ represents the time variable. 
Furthermore, for the fractional heat equation, 
\begin{equation*}\label{heat equation}
	\partial_t u + (-\Delta)^s u=0, \quad(t,x)\in \mathbb{R}_+\times\mathbb{R},
\end{equation*}	
in the early work\cite{G. Pólya}, G. Pólya 
first studied the asymptotic behavior of the fundamental solution by examining the term $|x|^{1+2s}K(1,x)$ as $x\rightarrow+\infty$:
\begin{equation*}
	\mathcal{K}(1,x)\lesssim \frac{1}{(1+|x|)^{1+2s}}.
\end{equation*}	
Subsequently, D. G. Aronson \cite{D. G. Aronson} established global bounds for the fundamental solution of uniformly parabolic equation. In 1987, E. B. Davies\cite{E. B. Davies} derived Gaussian upper bounds for the heat kernels through the use of logarithmic Sobolev inequalities. Since these results pertain to classical parabolic operators, A. Pascucci and S. Polidoro \cite{A. Pascucci}  further extended the analysis by proving an upper bound for the fundamental solution that is independent of the regularity of the coefficients. We also refer the reader to the works of \cite{R. M. Blumenthal}, \cite{E. B. Fabes}, \cite{J. Nash} for further contributions on this topic.

We now focus on the case $s=1$ of equation \eqref{aim equation}, specifically $\partial_t u+v\cdot\nabla_xu-\Delta_v u=0$.
In 1934, Kolmogorov\cite{A. Kolmogorov} provided a fundamental solution to this equation and demonstrated that the fundamental solution satisfies Gaussian bounds.
We are also interested in the fractional counterpart of the aforementioned equation. It is well known that  the fractional equation serves as a prototype for a
kinetic partial differential equation, bearing similarities to the renowned Boltzmann and Landau equations.
Recently, in \cite{A. Loher}, A. Loher  established a  polynomial upper bound for the fundamental solution of  kinetic integro-differential equations, with the fractional Kolmogorov equation $\partial_t u+v\cdot\nabla_x u=(-\Delta_v)^s u+h$ serving as a prototypical example of constant coefficient equations. Additionally, P. Auscher, C. Imbert and L. Niebel \cite{P. Auscher} investigated the construction of weak solutions for kinetic equations of Kolmogorov-Fokker-Planck type. Other significant contributions related to the Kolmogorov equation can be found in the works of \cite{B. Farkas},  \cite{C. Imbert}, \cite{L. P. Kuptsov}, \cite{E. Lanconell}, \cite{A. Lanconelli}, \cite{S. Polidoro}.

The analysis of properties of fundamental solutions to Cauchy problems like \eqref{aim equation} have far-reaching applications, including the study of well-posedness and regularity of solutions.
 In 1951, M. Weber \cite{M. Weber} generalized Kolmogorov's results to a broader class of equations and provided a uniqueness theorem for the fundamental solution. In \cite{Zimo Hao}, the authors  established the global well-posedness for the following Fokker-Planck equation with small initial value:
\begin{equation*}
	\partial_t u=(\Delta_v^{\frac{\alpha}{2}}-v\cdot\nabla_x)u -div_v((b*u)u),~u(0)=u_0.
\end{equation*}	
With regard to regularity, building on Kolmogorov's work, Hörmander \cite{L. Hörmander} examined the regularity properties of a much broader class of equations. L. Niebel and R. Zacher \cite{L. Niebel} introduced the notion of kinetic maximal $L^2$-regularity with temporal weights and demonstrated that the fractional Kolmogorov equation satisfies this regularity property. Later, in \cite{L. Niebel 0},  L. Niebel also
extended this result by proving that the Kolmogorov equation admits kinetic maximal $L^p_\mu$-regularity, provided certain conditions on the density, as well as on the parameters $p$ and $\mu$, are satisfied.

Motivated by the aforementioned works, the primary objective of this paper is to establish the pointwise upper bound for the fundamental solution of \eqref{aim equation}, 
which is expected to play a crucial role in future investigations of well-posedness. Our main result of this paper is stated as follows:
\begin{theorem}\label{estimate K}
	Let $\forall x,v,\xi,\eta \in \mathbb{R}$, then for any $b_1,b_2 \in \mathbb{N},~s\in(0,1)$, there exists $\varepsilon>0$ such that the fundamental solution $\mathcal{K}(1,x,v)$ of \eqref{aim equation} satisfies the following inequality:
	\begin{equation*}
		\left|\partial_x^{b_1}\partial_v^{b_2}\mathcal{K}(1,x,v)\right|
		\lesssim\frac{1}{\langle x,x+v \rangle ^{2+2s-2\varepsilon}\langle x\rangle^{\varepsilon+b_1}\langle x+v\rangle^{\varepsilon+b_2}},
	\end{equation*}
	where $\varepsilon$ arbitrarily small.
\end{theorem}
Thanks to the scale invariance property of the fundamental, namely $\mathcal{K}(t,x,v)=\frac{1}{t^{1+\frac{1}{s}}}\mathcal{K}(1,\frac{x}{t^{1+\frac{1}{2s}}},\frac{v}{t^{\frac{1}{2s}}})$, we can infer the following Corollary:
\begin{corollary}
	For any $b_1,b_2 \in \mathbb{N},~s\in(0,1)$, there exists $\varepsilon>0$ such that the fundamental solution $\mathcal{K}(t,x,v)$ of \eqref{aim equation} holds:
	\begin{equation*}
		\left|\partial_x^{b_1}\partial_v^{b_2}\mathcal{K}(t,x,v)\right|
		\lesssim
		\frac{t^{-1-\frac{2+b_1+b_2}{2s}-b_1}}{\langle \frac{x}{t^{1+\frac{1}{2s}}},\frac{x}{t^{1+\frac{1}{2s}}}+\frac{v}{t^\frac{1}{2s}} \rangle ^{2+2s-2\varepsilon}\langle \frac{x}{t^{1+\frac{1}{2s}}}\rangle^{\varepsilon+b_1}\langle \frac{x}{t^{1+\frac{1}{2s}}}+\frac{v}{t^\frac{1}{2s}}\rangle^{\varepsilon+b_2}},
	\end{equation*}
	where $\varepsilon$ arbitrarily small.
\end{corollary}	
	\begin{remark}
		While this work was being completed, the article \cite{Haojie Hou} was uploaded on arXiv. Therein, the authors established precise two-sided bounds for the fractional Kolmogorov equation in arbitrary dimensions, leveraging the underlying stochastic jump process.
		It is important to emphasize that the proof techniques employed in \cite{Haojie Hou} are fundamentally distinct from those presented in the current paper, and the results in our paper are independent of those in \cite{Haojie Hou}.
		While the proof in \cite{Haojie Hou} employs tools from stochastic analysis, our approach is based solely on the Littlewood-Paley decomposition. We present a concise and straightforward method to obtain the pointwise upper bound.
	\end{remark}

\section{Derivation of the fundamental solution}
In this section, we present the derivation of the fundamental solution of the equation \eqref{aim equation}.

For the linear part of the equation \eqref{aim equation}, we apply the space-velocity Fourier transform.
The Fourier transform in $(x,v)$ with respective Fourier variables $(\xi,\eta)$ of a function $f$ will be denoted by $\hat{f}$.
Thus we can rewrite it as follows:
\begin{equation*}
	\partial_t \hat{f} -\xi\cdot\nabla_\eta \hat{f}+|\eta|^{2s}\hat{f}=0.
\end{equation*}	
Fixing a value of $\eta\in \mathbb{R}$, we apply the method of characteristics for a fixed $\xi\in \mathbb{R}$. Along the characteristic path $\eta(t)=-\xi t+\eta$, we obtain the following result:
\begin{equation*}
	\partial_t \hat{f}(t,\xi,\eta(t))=-|\eta(t)|^{2s}\hat{f}(t,\xi,\eta(t)).
\end{equation*}
Therefore, it follows straightforwardly that the fundamental solution is expressed by
\begin{equation*}
	 f(t,x,v)=\mathcal{F}^{-1}(e^{-\int_{0}^{t} |\eta-\xi\tau|^{2s} d\tau}).
\end{equation*}	
By virtue of the scaling property,
we define the kernal $\mathcal{K}(1,x,v)=\mathcal{F}^{-1}(e^{-\int_{0}^{1} |\eta-\xi\tau|^{2s} d\tau})$.
Thus we have
\begin{equation*}
	\begin{split}
		\mathcal{K}(
		1,x,v)&=\iint_{\mathbb{R}\times\mathbb{R}} \exp\left(-\int_0^1 |\xi \tau-\eta|^{2s}d\tau - i\xi \cdot x -i \eta \cdot v\right)d\xi d\eta\\
		&=\iint_{\mathbb{R}\times\mathbb{R}} \exp\left(-\int_0^1 |(\xi+\eta) \tau-\eta|^{2s}d\tau - i(\xi+\eta) \cdot x -i \eta \cdot v\right)d\xi d\eta\\
		&=\iint_{\mathbb{R}\times\mathbb{R}} \exp\left(-\int_0^1 |(\xi-\eta) \tau+\eta|^{2s}d\tau - i(\xi-\eta) \cdot x +i \eta \cdot v\right)d\xi d\eta\\
		&=\iint_{\mathbb{R}\times\mathbb{R}} \exp\left(-M(\xi,\eta) - i\xi \cdot x -i\eta \cdot (-x-v)\right) d\xi d\eta.
	\end{split}
\end{equation*}
Where,
\begin{equation*}
	M(\xi,\eta)=\int_{0}^{1} |(\xi-\eta)\tau+\eta|^{2s} d\tau
	=\frac{1}{2s+1}\frac{|\xi|^{2s}\xi - |\eta|^{2s} \eta}{\xi-\eta}.
\end{equation*}

\section{Notations and Basic Lemmas}
Throughout the paper, we denote $A\lesssim B$ if there exists a universal constant C such that $A\leq CB$. 
To prove the main theorem, we presentx several preliminary lemmas.
\begin{lemma}\label{estimate of M}
	For $\forall s\in (0,1)$, 
	 $\forall m_1, m_2 \in \mathbb{N}$, we have the estimate
	\begin{equation*}
		\big|\partial_{\xi}^{m_1} \partial_{\eta}^{m_2} M(\xi,\eta)\big|\lesssim_{s,m_1,m_2} 
		\frac{|\xi|^{2s+1-m_1}+|\eta|^{2s+1-m_2}}{|\xi|^{m_1+1}+|\eta|^{m_2+1}}.
	\end{equation*}
\end{lemma}
\begin{proof}
	To derive the estimate, we split it into two cases based on the relationship between $|\xi-\eta|$ and $|\eta|$.
	Firstly, we commence with  the case when $|\xi-\eta| \leq \frac{1}{2}|\eta|$. Introducing
	\begin{equation*}
		M(\xi,\eta)=M_1(\xi,\eta)
		+R(\xi,\eta),
	\end{equation*}
	where
	\begin{equation*}
		\begin{split}
			M_1(\xi,\eta)&=\sum_{n=1}^{m_1+1} \binom{2s+1}{n} \eta^{2s+1-n} (\xi-\eta)^{n-1},\\
			R(\xi,\eta)	&=\frac{\eta^{2s+1}}{\xi-\eta}\int_{0}^{\frac{\xi-\eta}{\eta}} \partial_t^{m_1+2}\left((1+t)^{2s+1}\right) \left(\frac{\xi-\eta}{\eta}-t\right)^{m_1+1}dt\\
			&=\frac{(m_1+2)!\binom{2s+1}{m_1+2}}{\xi-\eta}\int_{\eta}^{\xi} t^{2s-m_1-2} (\xi-t)^{m_1+2} dt.
		\end{split}
	\end{equation*}
	Subsequently, we differentiate with respect to both $\xi$ and $\eta$ to get
	\begin{equation*}
		\begin{split}
			&\big|\partial_{\xi}^{m_1} \partial_{\eta}^{m_2} M_1(\xi,\eta)\big|
			=\left|\frac{{m_1}!}{2s+1}\binom{2s+1}{m_1+1} \partial_{\eta}^{m_2}  \left(\eta^{2s-m_1}\right)\right|
			\lesssim |\eta|^{2s-m_1-m_2},\\
			&\big|\partial_{\xi}^{m_1} \partial_{\eta}^{m_2} R(\xi,\eta)\big|
			\lesssim |\eta|^{2s-m_1-m_2},
		\end{split}
	\end{equation*}
	which indicates that
	\begin{equation*}\label{case 1}
		\big|\partial_{\xi}^{m_1} \partial_{\eta}^{m_2}M(\xi,\eta)\big|
		\lesssim (|\xi|+|\eta|)^{2s-m_1-m_2}.
	\end{equation*}
	On the other hand, for the case $|\xi-\eta|\geq \frac{1}{2}|\eta|$, two situations arise.
	On account of 
	\begin{equation*}
		\begin{split}
			&\partial_{\xi}^{m_1} \partial_{\eta}^{m_2} M(\xi,\eta)\\
			&=\partial_{\eta}^{m_2} \left(\frac{1}{2s+1} \sum_{k=0}^{m_1} \binom{m_1}{k} \left(|\xi|^{2s}\xi-|\eta|^{2s}\eta\right)^{(k)} \left(\dfrac{1}{\xi-\eta}\right)^{(m_1-k)}\right)\\
			&=\frac{1}{2s+1} \partial_{\eta}^{m_2} \left( (-1)^{m_1} m_1!	\frac{|\xi|^{2s}\xi-|\eta|^{2s}\eta}{(\xi-\eta)^{m_1+1}}
			+\sum_{k=1}^{m_1}\frac{(-1)^{m_1-k}m_1!}{k!}\frac{ \left(\sum_{j=0}^{k}\binom{k}{j} (|\xi|^{2s})^{(j)} \xi^{(k-j)}\right)
			}{(\xi-\eta)^{m_1-k+1}}\right)\\ 
			&=\frac{(-1)^{m_1+m_2}(m_1+m_2)!}{2s+1}	\dfrac{|\xi|^{2s}\xi-|\eta|^{2s}\eta}{(\xi-\eta)^{m_1+m_2+1}}
			+\sum_{j=1}^{m_2}\binom{m_2}{j}
			\frac{(-1)^{m_1+m_2-j}(m_1+m_2-j)!}{2s+1}
			 \frac{\left(\sum_{i=0}^{j}\binom{j}{i} (|\eta|^{2s})^{(i)} \eta^{(j-i)}\right)}{(\xi-\eta)^{m_1+m_2-j+1}}\\
			&\quad+\sum_{k=1}^{m_1}
			\frac{(-1)^{m_1+m_2-2k}}{2s+1}
			\frac{(m_1+m_2-k)!}{k!}\dfrac{\sum_{j=0}^{k}\binom{k}{j}(|\xi|^{2s})^{(j)}\xi^{(k-j)} }{(\xi-\eta)^{m_1+m_2-k+1}}.
		\end{split}
	\end{equation*}
	We calculate for the situation when $|\xi|\leq\frac{1}{4}|\eta|$,
	\begin{equation}\label{case 2}
		\begin{split}
		\big|\partial_{\xi}^{m_1} \partial_{\eta}^{m_2} M(\xi,\eta)\big|
		&\lesssim_{s,m_1,m_2}
		\dfrac{|\xi|^{2s+1}+|\eta|^{2s+1}}{|\eta|^{m_1+m_2+1}}
		+\sum_{j=1}^{m_2}\dfrac{|\eta|^{2s+1-j}}{|\eta|^{m_1+m_2-j+1}}
		+\sum_{k=1}^{m_1}\dfrac{|\xi|^{2s+1-k} }{|\eta|^{m_1+m_2-k+1}}\\
		&\lesssim_{s,m_1,m_2} |\eta|^{2s-m_1-m_2} +\frac{1}{|\eta|^{m_2+1}} \sum_{k=1}^{m_1} \frac{|\xi|^{2s+1-k}}{|\eta|^{m_1-k}}\\
		&\lesssim_{s,m_1,m_2}|\eta|^{2s-m_1-m_2}
		+ \frac{|\xi|^{2s+1-m_1}}{|\eta|^{m_2+1}}.
		\end{split}
	\end{equation}
	Next, our attention turns to the situation $|\xi|\geq 4|\eta|$. It is evident that $|\xi-\eta|\sim |\xi|$.
	Consequently, we infer
	\begin{equation*}\label{case 3}
		\begin{split}
		\big|\partial_{\xi}^{m_1} \partial_{\eta}^{m_2} M(\xi,\eta)\big|
		&\lesssim_{s,m_1,m_2}
		\dfrac{|\xi|^{2s+1}+|\eta|^{2s+1}}{|\xi|^{m_1+m_2+1}}
		+\sum_{j=1}^{m_2}\dfrac{|\eta|^{2s+1-j}}{|\xi|^{m_1+m_2-j+1}}
		+\sum_{k=1}^{m_1}\dfrac{|\xi|^{2s+1-k} }{|\xi|^{m_1+m_2-k+1}}\\
		&\lesssim_{s,m_1,m_2} |\xi|^{2s-m_1-m_2} +\frac{1}{|\xi|^{m_1+1}} \sum_{j=1}^{m_2} \frac{|\eta|^{2s+1-j}}{|\xi|^{m_2-j}}\\
		&\lesssim_{s,m_1,m_2}|\xi|^{2s-m_1-m_2}
		+ \frac{|\eta|^{2s+1-m_2}}{|\xi|^{m_1+1}}.
		\end{split}
	\end{equation*}
	Therefore, by integrating this with \eqref{case 2}, we obtain that for $|\xi-\eta|\geq \frac{1}{2}|\eta|$, the following holds
	\begin{equation*}
		\begin{split}
		&\big|\partial_{\xi}^{m_1} \partial_{\eta}^{m_2}M(\xi,\eta)\big| \lesssim_{s,m_1,m_2}|\eta|^{2s-m_1-m_2}
		+ \frac{|\xi|^{2s+1-m_1}}{|\eta|^{m_2+1}},\quad\text{for}\quad|\xi|\leq\frac{1}{4}|\eta|,\\
		&\big|\partial_{\xi}^{m_1} \partial_{\eta}^{m_2}M(\xi,\eta)\big| \lesssim_{s,m_1,m_2}|\xi|^{2s-m_1-m_2}
		+ \frac{|\eta|^{2s+1-m_2}}{|\xi|^{m_1+1}},\quad\text{for}\quad|\xi|\geq 4|\eta|.
		\end{split}
	\end{equation*}
	In conclusion, we can deduce the estimate
	\begin{equation*}
		\begin{split}
		\big|\partial_{\xi}^{m_1} \partial_{\eta}^{m_2}M(\xi,\eta)\big| &\lesssim_{s,m_1,m_2}
		|\eta|^{2s-m_1-m_2}\min\left\{\frac{|\eta|^{1+m_1}}{|\xi|^{1+m_1}},1\right\}
		+
		|\xi|^{2s-m_1-m_2}\min\left\{\frac{|\xi|^{1+m_2}}{|\eta|^{1+m_2}},1\right\}\\
		&\sim\frac{|\xi|^{1+2s-m_1}+|\eta|^{1+2s-m_2}}{|\xi|^{m_1+1}+|\eta|^{m_2+1}}.
		\end{split}
	\end{equation*}
	This completes the proof of Lemma \ref{estimate of M}.
\end{proof}
Next Lemma is about an estimate for the derivatives of $e^{-M(\xi,\eta)}$.
\begin{lemma}\label{lemma e^M}
	For $\forall m_1,m_2 \in \mathbb{Z^+}$, we have 
	\begin{equation}\label{estimate e^M}
		\big|\partial_{\xi}^{m_1} \partial_{\eta}^{m_2} e^{-M(\xi,\eta)}\big|
		\lesssim_{s,m_1,m_2}\frac{|\xi|^{1+2s-m_1}+|\eta|^{1+2s-m_2}}{|\xi|^{m_1+1}+|\eta|^{m_2+1}}.
	\end{equation}
\end{lemma}
\begin{proof}
	To proceed with the proof, we start by introducing Francesco Faà di Bruno's formula,
	\begin{equation*}
		\frac{d^n}{dx^n} f(g(x))=\sum_{\sum_{p=1}^{n}pm_p=n} \frac{n!}{m!} f^{|m|}\left(g(x)\right)\prod_{j=1}^{n}\left(g^{(j)}(x)\right)^{m_j}.
	\end{equation*}	
	Utilizing the above formula, we can derive
	\begin{equation}\label{formula}
		\begin{split}
		&\partial_\xi^{m_1} \partial_\eta^{m_2} e^{-M(\xi,\eta)}\\
		&=\sum_{\substack{\sum_{j=1}^{m_1} j\alpha_j=m_1\\0\leq k \leq m_2}}  C_{\alpha,k}
		\partial_\eta^k e^{-M(\xi,\eta)}
		\partial_\eta^{m_2-k}
		(
		\prod_{j=1}^{m_1}(\partial_\xi^j (-M(\xi,\eta)))^{\alpha_j}
		)\\
		&=\sum_{\substack{\sum_{j=1}^{m_1} j\alpha_j=m_1\\0\leq k\leq m_2\\\sum_{l=1}^{k} l\beta_l=k}}    
		C_{\alpha,\beta,k,\gamma} e^{-M(\xi,\eta)}
		\sum_{\substack{\sum_{j=1}^{m_1} \gamma_j=m_2-k\\\sum_{i=1}^{\alpha_j}\mu_i=\gamma_j}}
		\prod_{l=1}^{k} (\partial_\eta^l (-M(\xi,\eta)))^{\beta_l}
		\prod_{j=1}^{m_1} \prod_{i=1}^{\alpha_j} \partial_\eta^{\mu_i} \partial_\xi^j M(\xi,\eta).
		\end{split}
	\end{equation}
	Now, to apply Lemma \ref{estimate of M}, we shall split it into three cases.
	Considering the first case, where $|\xi-\eta|\leq \frac{1}{2} |\eta|$, the computation proceeds as follows.
	\begin{equation*}
		\begin{split}
		&\big|\partial_\eta^{m_2} \partial_\xi^{m_1} e^{-M(\xi,\eta)}\big|\\
		&\lesssim \sum_{\substack{\sum_{j=1}^{m_1} j\alpha_j=m_1\\\sum_{l=1}^{k} l\beta_l=k}}  \sum_{k=0}^{m_2}  \big|e^{-M(\xi,\eta)}\big| \prod_{l=1}^{k} 
		\left(|\xi|+|\eta| \right)^{(2s-l)\beta_l}
		\sum_{\sum_{j=1}^{m_1} \gamma_j=m_2-k}
		\sum_{\mu_1+\cdots+\mu_{\alpha_j}=\gamma_j}
		\prod_{j=1}^{m_1}
		\prod_{i=1}^{\alpha_j}
		\left(|\xi|+|\eta \right)^{2s-\mu_i-j}\\
		&\lesssim \sum_{\substack{\sum_{j=1}^{m_1} j\alpha_j=m_1\\\sum_{l=1}^{k} l\beta_l=k}}  \sum_{k=0}^{m_2}  \big|e^{-M(\xi,\eta)}\big| \left(|\xi|+|\eta| \right)^{\sum_l (2s-l)\beta_l + \sum_{j=1}^{m_1} \sum_{i=1}^{\alpha_j} \sum_{\gamma}\sum_{\mu}(2s-\mu_i-j)}\\
		&\lesssim \sum_{\substack{\sum_{j=1}^{m_1} j\alpha_j=m_1\\\sum_{l=1}^{k} l\beta_l=k}}  \sum_{k=0}^{m_2}   \big|e^{-M(\xi,\eta)}\big| \left(|\xi|+|\eta| \right)^{2s(|\beta|+|\alpha|)-m_1-m_2}\\
		&\lesssim \left( |\xi|+
		|\eta|\right)^{2s-m_1-m_2}.
		\end{split}
	\end{equation*}
	Next we consider $|\xi-\eta|\geq \frac{1}{2} |\eta|$. It consists of two cases.\\
	For the case $|\xi|\leq \frac{1}{4}|\eta|$, we apply Lemma \ref{estimate e^M} and \eqref{formula} once more to achieve
	\begin{equation*}
		\begin{split}
		\big|\partial_\eta^{m_2} \partial_\xi^{m_1} e^{-M(\xi,\eta)}\big|
		&\lesssim \sum_{\substack{\sum_{j=1}^{m_1} j\alpha_j=m_1\\\sum_{l=1}^{k} l\beta_l=k}}  \sum_{k=0}^{m_2} e^{-M(\xi,\eta)} |\eta|^{2s|\beta|-k} \frac{|\eta|^{(2s+1)|\alpha|-m_1} + |\xi|^{(2s+1)|\alpha|-m_1}}{|\eta|^{|\alpha|+m_2-k}}\\
		&\lesssim\sum_{\substack{\sum_{j=1}^{m_1} j\alpha_j=m_1\\\sum_{l=1}^{k} l\beta_l=k}}  \sum_{k=0}^{m_2}  e^{-M(\xi,\eta)} \frac{|\eta|^{(2s+1)|\alpha|-m_1}+|\xi|^{(2s+1)|\alpha|-m_1}}{|\eta|^{|\alpha|+m_2-2s|\beta|}}\\
		&\lesssim\sum_{\substack{\sum_{j=1}^{m_1} j\alpha_j=m_1\\\sum_{l=1}^{k} l\beta_l=k}}  \sum_{k=0}^{m_2}  \frac{1}{(|\xi|+|\eta|)^{2s(|\alpha|+|\beta|-1)}} \frac{|\eta|^{(2s+1)|\alpha|-m_1}+|\xi|^{(2s+1)|\alpha|-m_1}}{|\eta|^{|\alpha|+m_2-2s|\beta|}}\\
		&\lesssim (|\eta|^{2s-m_1-m_2} + \frac{|\xi|^{2s+1-m_1}}{|\eta|^{m_2+1}}).
		\end{split}
	\end{equation*}
	Similarly, for $|\xi|\geq 4|\eta|$, it is straightforward to verify
	\begin{equation*}
		\big|\partial_{\xi}^{m_1} \partial_{\eta}^{m_2}e^{-M(\xi,\eta)}\big| \lesssim_{s,m_1,m_2}|\xi|^{2s-m_1-m_2}
		+ \frac{|\eta|^{2s+1-m_2}}{|\xi|^{m_1+1}}.
	\end{equation*}
	Therefore, based on the aforementioned cases, it can be concluded that
	\begin{equation*}
		\big|\partial_{\xi}^{m_1} \partial_{\eta}^{m_2}e^{-M(\xi,\eta)}\big| \lesssim_{s,m_1,m_2}
		\frac{|\xi|^{1+2s-m_1}+|\eta|^{1+2s-m_2}}{|\xi|^{m_1+1}+|\eta|^{m_2+1}}.
	\end{equation*}
	This completes the proof of Lemma \ref{lemma e^M}.
\end{proof}
Now, to simplify the outcomes, we introduce
\begin{equation*}
	\begin{split}
	G_{m_1,m_2}
	&:=\sum_{a=0}^{b_2}
	\frac{C_{a,b_2}}{v^{2+b_2-a}} \iint_{\mathbb{R}\times\mathbb{R}} \chi\left(\frac{|\xi|}{2^{m_1}}\right)\chi\left(\frac{|\eta|}{2^{m_2}}\right) (-i\xi)^{b_1} (-i\eta)^a \partial_\eta^2 e^{-M(\xi,\eta)} e^{-i\xi\cdot x-i\eta\cdot v}  d\xi d\eta,\\
	H_{m_1,m_2}
	&:= \sum_{m_1,m_2}\sum_{a=0}^{b_2}
	\frac{C_{a,b_2}}{v^{3+b_2-a}} \iint_{\mathbb{R}\times\mathbb{R}} \chi\left(\frac{|\xi|}{2^{m_1}}\right)\chi\left(\frac{|\eta|}{2^{m_2}}\right) (-i\xi)^{b_1} (-i\eta)^a \partial_\eta^3e^{-M(\xi,\eta)} e^{-i\xi\cdot x-i\eta\cdot v}  d\xi d\eta.
	\end{split}
\end{equation*}

In what follows, we always denote
\begin{equation*}
	\widetilde{K}_{m_1,m_2}:=\begin{cases}
		\left|	G_{m_1,m_2}\right|\quad &\text{for} \qquad 0<s\leq \frac{1}{2},\\
		\left|H_{m_1,m_2}\right|  \qquad &\mbox{for} \qquad \frac{1}{2}<s<1,
	\end{cases}
\end{equation*}
\begin{equation*}
	s^*=\begin{cases}
		s \quad &\text{for} \qquad 0<s\leq \frac{1}{2},\\
		s-\frac{1}{2} \qquad &\mbox{for} \qquad \frac{1}{2}<s<1,
	\end{cases}
\end{equation*}
and
\begin{equation*}
	\boxplus=\begin{cases}
		2 \quad &\text{for} \qquad 0<s\leq \frac{1}{2},\\
		3 \qquad &\mbox{for} \qquad \frac{1}{2}<s<1.
	\end{cases}
\end{equation*}
Subsequently, we establish the following Lemma, which is crucial for the proof of Theorem \ref{estimate K}.
\begin{lemma}\label{estimate K_m1m2}
	For $m_1,m_2 \in \mathbb{Z}$,  we have following estimates:	
	\begin{align}
		\widetilde{K}_{m_1,m_2} 
		&\lesssim \frac{2^{m_1(2s^*+b_1)}}{|v|^{\boxplus+b_2}} \left(1+(|v|2^{m_1})^{b_2}\right)\min\left\{\frac{2^{-m_1}}{|x|^{\frac{n_1}{n_1+n_2}}|v|^{\frac{n_2}{n_1+n_2}}}  ,1\right\}^{n_1+n_2}, ~|m_1-m_2|\leq2,\label{h}\\
		\widetilde{K}_{m_1,m_2} 
		&\lesssim  \frac{2^{m_1(1+b_1)+m_2(2s^*-1)}}{|v|^{\boxplus+b_2}} \left(1+(|v|2^{m_2})^{b_2}\right)
		\min\left\{\frac{2^{-m_2}}{|v|},1\right\}^{n_3}, ~  m_2\geq m_1+3,\label{j}\\
		\widetilde{K}_{m_1,m_2} 
		&\lesssim  \frac{2^{m_22s^*}}{|x|^{b_1}|v|^{\boxplus+b_2}}
		\left(1+(|v|2^{m_2})^{b_2}\right) 
		\min\left\{\frac  {2^{-m_1n_4-m_2n_5}}{|x|^{n_4}|v|^{n_5}},1\right\}, ~m_2\leq m_1-3.\label{k}
	\end{align}
\end{lemma}	
\begin{proof}
	In order to acquire the estimates, we categorize the analysis into two cases based on the range of $s$. Applying Lemma \ref{estimate e^M}.\\
	\textbf{case $1: 0<s\leq \frac{1}{2}$}. Each case is addressed as follows.\\
	For $|m_1-m_2|\leq2$, thanks to \eqref{estimate e^M}, we have
	\begin{equation*}
		\begin{split}
		&\left|G_{m_1,m_2}\right|\\
		&\lesssim \sum_{0\leq a\leq b_2} \frac{1}{|x|^{n_1}|v|^{n_2+2+b_2-a}}\iint_{\substack{|\xi|\sim2^{m_1}\\|\eta|\sim2^{m_2}}} 
		\left|
		\partial_\xi^{n_1} \partial_\eta^{n_2} 
		\left(
		(-i\xi)^{b_1} (-i\eta)^a \partial_\eta^2 e^{-M(\xi,\eta)}
		\chi\left(\frac{|\xi|}{2^{m_1}}\right)\chi\left(\frac{|\eta|}{2^{m_2}}\right)
		\right)
		\right| d\xi d\eta\\
		&\lesssim \sum_{0\leq a\leq b_2} \frac{1}{|x|^{n_1}|v|^{n_2+2+b_2-a}}\iint_{\substack{|\xi|\sim2^{m_1}\\|\eta|\sim2^{m_2}}}  \sum_{\substack{0\leq j+k\leq n_1\\0\leq l+r\leq n_2}}
		\left|\partial_\xi^k \partial_\eta^{r+2} e^{-M(\xi,\eta)}
		\right|
		\left|\partial_\xi^j (-i\xi)^{b_1}
		\right|
		\left|\partial_\eta^l (-i\eta)^a
		\right|\\
		&\qquad\qquad\qquad\qquad\qquad\qquad\qquad\qquad\qquad\qquad\left|\partial_\xi^{n_1-j-k} \chi\left(\frac{|\xi|}{2^{m_1}}\right)
		\right|
		\left|\partial_\eta^{n_2-r-l} \chi\left(\frac{|\eta|}{2^{m_2}}\right)
		\right| d\xi d\eta\\
		&\lesssim \sum_{\substack{0\leq a\leq b_2\\0\leq j+k\leq n_1\\0\leq l+r\leq n_2}} \frac{1}{|x|^{n_1}|v|^{n_2+2+b_2-a}}\iint_{\substack{|\xi|\sim2^{m_1}\\|\eta|\sim2^{m_2}}}  \left(|\xi|+|\eta|\right)^{2s-k-r-2}|\xi|^{b_1-j} |\eta|^{a-l} 2^{-m_1(n_1-j-k)} 2^{-m_2(n_2-r-l)} d\xi d\eta\\
		&\lesssim \sum_{0\leq a\leq b_2} \frac{1}{|x|^{n_1}|v|^{n_2+2+b_2-a}}2^{m_1(2s+b_1+a-n_1-n_2)}. 
		\end{split}
	\end{equation*}
	Specially, by choosing $n_1=0, n_2=0$, we obtain
	\begin{equation*}
		\left|G_{m_1,m_2}\right| \lesssim \sum_{0\leq a\leq b_2} \frac{1}{|v|^{2+b_2-a}}2^{m_1(2s+b_1+a)}. 
	\end{equation*}
	Thus the aforementioned two estimates suggest that
	\begin{equation*}
		\begin{split}
		\left|G_{m_1,m_2}\right|& \lesssim \sum_{0\leq a\leq b_2} \frac{1}{|v|^{2+b_2-a}}2^{m_1(2s+b_1+a)} \min\{\frac{1}{|x|^{n_1}|v|^{n_2}}  2^{-m_1(n_1+n_2)},1\}\\
		&\lesssim \frac{2^{{m_1}(2s+b_1)}}{|v|^{2+b_2}} (1+(|v|2^{m_1})^{b_2}) \min\{\frac{2^{-m_1}}{|x|^{\frac{n_1}{n_1+n_2}} |v|^{\frac{n_2}{n_1+n_2}} },1\}^{n_1+n_2}.
		\end{split}
	\end{equation*}
	We now demonstrate that for $m_2\geq m_1+3$, we derive the following results:
	\begin{equation*}
		\begin{split}
		&\left|G_{m_1,m_2}\right|\\
		&\lesssim \sum_{0\leq a\leq b_2} \frac{1}{|v|^{n_3+2+b_2-a}}\iint_{\substack{|\xi|\sim2^{m_1}\\|\eta|\sim2^{m_2}}} 
		\left|
		\partial_\eta^{n_3} 
		\left(
		(-i\eta)^a \partial_\eta^2 e^{-M(\xi,\eta)}
		\chi\left(\frac{|\eta|}{2^{m_2}}\right)
		\right)
		(-i\xi)^{b_1}
		\chi\left(\frac{|\xi|}{2^{m_1}}\right)
		\right| d\xi d\eta\\
		&\lesssim \sum_{0\leq a\leq b_2} \frac{1}{|v|^{n_3+2+b_2-a}}\iint_{\substack{|\xi|\sim2^{m_1}\\|\eta|\sim2^{m_2}}}  \sum_{0\leq l+r\leq n_3}
		\left| \partial_\eta^{r+2} e^{-M(\xi,\eta)}
		\right|
		\left|\partial_\eta^l (-i\eta)^a
		\right|
		\left|\partial_\eta^{n_3-r-l} \chi\left(\frac{|\eta|}{2^{m_2}}\right)
		\right|
		|\xi|^{b_1} d\xi d\eta\\
		&\lesssim \sum_{\substack{0\leq a\leq b_2\\0\leq l+r\leq n_3}} \frac{1}{|v|^{n_3+2+b_2-a}}\iint_{\substack{|\xi|\sim2^{m_1}\\|\eta|\sim2^{m_2}}}  \left(|\eta|^{2s-r-2}+\frac{|\xi|^{2s+1}}{|\eta|^{r+3}}\right)|\xi|^{b_1} |\eta|^{a-l}  2^{-m_2(n_3-r-l)} d\xi d\eta\\
		&\lesssim \sum_{0\leq a\leq b_2} \frac{1}{|v|^{n_3+2+b_2-a}}2^{m_1(1+b_1)+m_2(2s+a-n_3-1)}. 
		\end{split}
	\end{equation*}
	Specially, by choosing $n_3=0$, we get
	\begin{equation*}
		\left|G_{m_1,m_2}\right| \lesssim \sum_{0\leq a\leq b_2} \frac{1}{|v|^{2+b_2-a}} 2^{m_1(1+b_1)+m_2(2s+a-1)}. 
	\end{equation*}
	Hence, by synthesizing the aforementioned estimates, we deduce
	\begin{equation*}
		\begin{split}
		\left|G_{m_1,m_2}\right| &\lesssim \sum_{0\leq a\leq b_2} \frac{1}{|v|^{2+b_2-a}}2^{m_1(1+b_1)+m_2(2s+a-1)} \min\{\frac{1}{|v|^{n_3}}  2^{-m_2n_3},1\}\\
		&\lesssim\frac{2^{m_1(1+b_1)+m_2(2s-1)}}{|v|^{2+b_2}} (1+(|v|2^{m_2})^{b_2}) \min\{\frac{2^{-m_2}}{|v|},1\}^{n_3}.
		\end{split}
	\end{equation*}
	To prove the remaining bound $m_2\leq m_1-3$, we take the following stvarepsilons
	\begin{equation*}
		\begin{split}
			&\left|G_{m_1,m_2}\right|\\
			&\lesssim \sum_{0\leq a\leq b_2} \frac{1}{|x|^{n_4+b_1}|v|^{n_5+2+b_2-a}}\iint_{\substack{|\xi|\sim2^{m_1}\\|\eta|\sim2^{m_2}}} 
			\left|
			\partial_\xi^{n_4+b_1} \partial_\eta^{n_5} 
			\left(
			(-i\xi)^{b_1} (-i\eta)^a \partial_\eta^2 e^{-M(\xi,\eta)}
			\chi\left(\frac{|\xi|}{2^{m_1}}\right)\chi\left(\frac{|\eta|}{2^{m_2}}\right)
			\right)
			\right| d\xi d\eta\\
			&\lesssim \sum_{0\leq a\leq b_2} \frac{1}{|x|^{n_4+b_1}|v|^{n_5+2+b_2-a}}\iint_{\substack{|\xi|\sim2^{m_1}\\|\eta|\sim2^{m_2}}}  \sum_{\substack{0\leq j+k\leq n_4+b_1\\0\leq l+r\leq n_5}}
			\left|\partial_\xi^k \partial_\eta^{r+2} e^{-M(\xi,\eta)}
			\right|
			\left|\partial_\xi^j (-i\xi)^{b_1}
			\right|
			\left|\partial_\eta^l (-i\eta)^a
			\right|\\
			&\qquad\qquad\qquad\qquad\qquad\qquad\qquad\qquad\qquad\left|\partial_\xi^{n_4+b_1-j-k} \chi\left(\frac{|\xi|}{2^{m_1}}\right)
			\right|
			\left|\partial_\eta^{n_5-r-l} \chi\left(\frac{|\eta|}{2^{m_2}}\right)
			\right| d\xi d\eta\\
			&\lesssim \sum_{\substack{0\leq a\leq b_2\\0\leq j+k\leq n_4+b_1\\0\leq l+r\leq n_5}} \frac{1}{|x|^{n_4+b_1}|v|^{n_5+2+b_2-a}}\iint_{\substack{|\xi|\sim2^{m_1}\\|\eta|\sim2^{m_2}}}  (|\xi|^{2s-k-r-2}+\frac{|\eta|^{2s-1-r}}{|\xi|^{k+1}})|\xi|^{b_1-j} |\eta|^{a-l}\\ &\qquad\qquad\qquad\qquad\qquad\qquad\qquad\qquad\qquad\qquad\quad2^{-m_1(n_4+b_1-j-k)} 2^{-m_2(n_5-r-l)} d\xi d\eta\\
			&\lesssim \sum_{0\leq a\leq b_2} \frac{1}{|x|^{n_4+b_1}|v|^{n_5+2+b_2-a}}\left( 2^{m_1(2s-r-n_4-1)}2^{m_2(a+1-n_5+r)} 
			+2^{-m_1n_4} 2^{m_2(a+2s-n_5)}\right)\\
			&\lesssim \sum_{0\leq a\leq b_2} \frac{1}{|x|^{n_4+b_1}|v|^{n_5+2+b_2-a}} 2^{-m_1n_4-m_2n_5}
			2^{m_2(a+2s)}.
		\end{split}
	\end{equation*}
	Specially, choosing $n_4=0,n_5=0$, we can obtain
	\begin{equation*}
		\left|G_{m_1,m_2}\right| \lesssim \sum_{0\leq a\leq b_2} \frac{1}{|x|^{b_1}|v|^{2+b_2-a}} 
		2^{m_2(a+2s)}.
	\end{equation*}
	Then there holds
	\begin{equation*}
		\begin{split}
		\left|G_{m_1,m_2}\right| &\lesssim \sum_{0\leq a\leq b_2} \frac{1}{|x|^{b_1}|v|^{2+b_2-a}} 
		2^{m_2(a+2s)}\min\left\{\frac  {2^{-m_1n_4-m_2n_5}}{|x|^{n_4}|v|^{n_5}},1\right\}\\
		&\lesssim \frac{2^{m_22s}}{|x|^{b_1}|v|^{2+b_2}}
		(1+(|v|2^{m_2})^{b_2}) \min\{\frac{2^{-m_1n_4-m_2n_5}}{|x|^{n_4}|v|^{n_5}},1\}.
		\end{split}
	\end{equation*}
	\textbf{case $2: \frac{1}{2}<s\leq 1$}.
	The reasoning for $\frac{1}{2}<s<1$ is completely analogous, it can be readily verified
	\begin{equation*}
		\begin{split}
		&\left|H_{m_1,m_2}\right| 
		\lesssim \frac{2^{m_1(2s+b_1-1)}}{|v|^{3+b_2}} \left(1+(|v|2^{m_1})^{b_2}\right)\min\left\{\frac{2^{-m_1}}{|x|^{\frac{n_6}{n_6+n_7}}|v|^{\frac{n_7}{n_6+n_7}}}  ,1\right\}^{n_6+n_7},~|m_1-m_2|\leq2,\\
		&\left|H_{m_1,m_2}\right| \lesssim  \frac{2^{m_1(1+b_1)+m_2(2s-2)}}{|v|^{3+b_2}} \left(1+(|v|2^{m_2})^{b_2}\right)
		\min\left\{\frac{2^{-m_2}}{|v|},1\right\}^{n_8},~  m_2\geq m_1+3,\\
		&\left|H_{m_1,m_2}\right| \lesssim  \frac{2^{m_2(2s-1)}}{|x|^{b_1}|v|^{3+b_2}}
		\left(1+(|v|2^{m_2})^{b_2}\right) 
		\min\left\{\frac  {2^{-m_1n_9-m_2n_{10}}}{|x|^{n_9}|v|^{n_{10}}},1\right\},~m_2\leq m_1-3.
		\end{split}
	\end{equation*}
	We omit the details here.\\
	As a result, it comes out that for all $0<s<1$ and $m_1,m_2 \in \mathbb{Z}$,
	\begin{equation*}
		\begin{split}
		\widetilde{K}_{m_1,m_2} 
		&\lesssim \frac{2^{m_1(2s^*+b_1)}}{|v|^{\boxplus+b_2}} \left(1+(|v|2^{m_1})^{b_2}\right)\min\left\{\frac{2^{-m_1}}{|x|^{\frac{n_1}{n_1+n_2}}|v|^{\frac{n_2}{n_1+n_2}}}  ,1\right\}^{n_1+n_2}, ~|m_1-m_2|\leq2,\\
		\widetilde{K}_{m_1,m_2} 
		&\lesssim  \frac{2^{m_1(1+b_1)+m_2(2s^*-1)}}{|v|^{\boxplus+b_2}} \left(1+(|v|2^{m_2})^{b_2}\right)
		\min\left\{\frac{2^{-m_2}}{|v|},1\right\}^{n_3}, ~  m_2\geq m_1+3,\\
		\widetilde{K}_{m_1,m_2} 
		&\lesssim  \frac{2^{m_22s^*}}{|x|^{b_1}|v|^{\boxplus+b_2}}
		\left(1+(|v|2^{m_2})^{b_2}\right) 
		\min\left\{\frac  {2^{-m_1n_4-m_2n_5}}{|x|^{n_4}|v|^{n_5}},1\right\}, ~m_2\leq m_1-3.
		\end{split}
	\end{equation*}
	Hence the statements in Lemma \ref{estimate K_m1m2} are proved.
\end{proof}	

\section{Proof of main Theorem}
With the above estimates, we are in a position to complete the proof of Theorem \ref{estimate K}.
\begin{proof}[Proof of Theorem \ref{estimate K}]
	To derive the results, we partition our proof into three distinct cases.\\
	\textbf{case $1: |x|,|v|\leq 1 $}. By applying Lemma \ref{lemma e^M} and performing straightforward computations, it is evident that
	\begin{equation*}
		\begin{split}
		\left|\partial_x^{b_1} \partial_v^{b_2}{K}(1,x,v)\right|
		&=\left|\iint (i\xi)^{b_1} (i\eta)^{b_2} \exp\left(-M(\xi,\eta)-i\xi \cdot x-i\eta \cdot v\right) d\xi d\eta\right|\\
		&\lesssim \iint |\xi|^{b_1} |\eta|^{b_2} \left|\exp\left(-M(\xi,\eta\right)
		\right| d\xi d\eta\\
		&= \iint |\xi|^{b_1} |\eta|^{b_2} e^{(-(|\xi|+|\eta|)^{2s})} d\xi d\eta\lesssim 1.
		\end{split}
	\end{equation*}
	\textbf{case $2: |v|\geq|x|$}. 
	In this situation, it is necessary to divide into two cases based on the range of $s$. Initially, we make an observation regarding $0<s\leq \frac{1}{2}$.\\
	\textbf{case $2.1: 0<s\leq \frac{1}{2}$}.
	We apply the integration by parts procedure two times with respect to $\eta$, resulting in
	\begin{equation*}
		K(1,x,v)=\frac{1}{(-iv)^2} \iint \partial_\eta^2 e^{-M(\xi,\eta)} e^{-i\xi\cdot x-i\eta\cdot v}  d\xi d\eta,
	\end{equation*}	
	then we can derive the derivatives with respect to x and v 
	\begin{equation*}
		\partial_x^{b_1} \partial_v^{b_2} K(1,x,v)
		= \sum_{m_1,m_2}\sum_{a=0}^{b_2}
		\frac{C_{a,b_2}}{v^{2+b_2-a}} \iint (-i\xi)^{b_1} (-i\eta)^a \partial_\eta^2 e^{-M(\xi,\eta)} e^{-i\xi\cdot x-i\eta\cdot v}  d\xi d\eta
		=:\sum_{m_1,m_2} G_{m_1,m_2}.
	\end{equation*}
	Thanks to \eqref{h}, one has
	\begin{equation*}
		\begin{split}
		&\sum_{|m_1-m_2|\leq 2} \left|G_{m_1,m_2}\right|\\
		&\lesssim \sum_{0\leq a\leq b_2}
		\sum_{2^{-m_1(n_1+n_2)}\leq|x|^{n_1}|v|^{n_2}} 
		\frac{2^{m_1(2s+b_1+a-n_1-n_2)}}{|x|^{n_1}|v|^{n_2+2+b_2-a}}
		+\sum_{0\leq a\leq b_2}
		\sum_{2^{-m_1(n_1+n_2)}\geq|x|^{n_1}|v|^{n_2}} 
		\frac{2^{m_1(2s+b_1+a)}}{|v|^{2+b_2-a}}\\
		&\lesssim \sum_{0\leq a\leq b_2} \frac{1}{|x|^{n_1}|v|^{n_2+2+b_2-a}}\left(\frac{1}{|x|^{n_1}|v|^{n_2}}\right)^{\frac{2s+b_1+a-n_1-n_2}{n_1+n_2}}+\sum_{0\leq a\leq b_2} \frac{1}{|v|^{2+b_2-a}}\left(\frac{1}{|x|^{n_1}|v|^{n_2}}\right)^{\frac{2s+b_1+a}{n_1+n_2}}\\
		&\lesssim \frac{1}{|v|^{2+2s+b_1+b_2}} \sum_{0\leq a\leq b_2} \left(\frac{|v|}{|x|}\right)^{\frac{n_1}{n_1+n_2}(2s+b_1+a)}\\
		&\lesssim \frac{1}{|v|^{2+2s+b_1+b_2-\frac{n_1}{n_1+n_2}(2s+b_1+b_2)}} \frac{1}{|x|^{\frac{n_1}{n_1+n_2}(2s+b_1+b_2)}}\\
		&\sim\frac{1}{|v|^{2+2s+b_1+b_2-\varepsilon}} \frac{1}{|x|^\varepsilon},
		\end{split}
	\end{equation*}
	the claim follows by choosing $n_1+n_2>2s+b_1+b_2$ appropriately.\\
	Moreover, by virtue of \eqref{j}, one can achieve
	\begin{equation*}
		\begin{split}
		\sum_{m_2\geq m_1+3} \left|G_{m_1,m_2}\right|
		&\lesssim \sum_{0\leq a\leq b_2}
		\sum_{m_2=-\infty}^{+\infty} \sum_{m_1=-\infty}^{m_2-3}
		\frac{2^{m_1(1+b_1)} 2^{m_2(2s+a-1)}}{|v|^{2+b_2-a}}  \min\{\frac{2^{-m_2n_3}}{|v|^{n_3}} ,1\}\\
		&\lesssim \sum_{0\leq a\leq b_2} \sum_{2^{-m_2n_3}\leq|v|^{n_3}} \frac{2^{m_2(2s+a+b_1-n_3)}}{|v|^{n_3+2+b_2-a}} +\sum_{0\leq a\leq b_2} \sum_{2^{-m_2n_3}\geq|v|^{n_3}} \frac{2^{m_2(2s+a+b_1)}}{|v|^{2+b_2-a}} \\
		&\lesssim \sum_{0\leq a\leq b_2}  \frac{1}{|v|^{2+b_2-a+n_3}} \left(\frac{1}{|v|^{n_3}}\right)^{\frac{2s+a+b_1-n_3}{n_3}} 
		+\sum_{0\leq a\leq b_2}  \frac{1}{|v|^{2+b_2-a}} \left(\frac{1}{|v|^{n_3}}\right)^{\frac{2s+a+b_1}{n_3}} \\
		&\lesssim  \sum_{0\leq a\leq b_2} \frac{1}{|v|^{2+b_2-a}} \left(\frac{1}{|v|^{n_3}}\right)^{\frac{2s+a+b_1}{n_3}} \\
		&\sim\frac{1}{|v|^{2+2s+b_1+b_2}},
		\end{split}
	\end{equation*}
	the claim follows by choosing $n_3>2s+b_1+b_2$ appropriately.\\
	To simplify the remaining case: $m_2\leq m_1-3$, we introduce the notation:
	\begin{equation*}
		\begin{split}
		\frac{1}{n_4+n_5}\left(\log_2\left(\frac{1}{|x|^{n_4}|v|^{n_5}}\right)+3n_5\right)
		=:A,\\
		\frac{1}{n_5}\left(\log_2\left(\frac{1}{|x|^{n_4}|v|^{n_5}}\right)-m_1n_4\right) 
		=:B,\\
		\frac{1}{n_4+n_5}\left(\log_2\left(\frac{1}{|x|^{n_4}|v|^{n_5}}\right)-3n_4\right)
		=:D,\\
		\frac{1}{n_4}\left(\log_2\left(\frac{1}{|x|^{n_4}|v|^{n_5}}\right)-m_2n_5\right)
		=:E.
		\end{split}
	\end{equation*}
	On account of \eqref{k}, one can yield 
	\begin{equation}
		\begin{split}
		&\sum_{m_2\leq m_1-3} \left|G_{m_1,m_2}\right|\\
		&\quad\lesssim \sum_{m_2\leq m_1-3}
		\sum_{0\leq a\leq b_2}
		\frac{1}{|x|^{n_4+b_1}|v|^{n_5+2+b_2-a}}\mathbf{1}_
		{
			2^{-m_1n_4-m_2n_5}	\leq 
			{|x|^{n_4}|v|^{n_5}}
		}
		2^{-m_1n_4} 2^{m_2(2s+a-n_5)} \\
		&\quad\quad+\sum_{m_2\leq m_1-3}\sum_{0\leq a\leq b_2}
		\frac{1}{|x|^{b_1}|v|^{2+b_2-a}}\mathbf{1}_
		{
			2^{-m_1n_4-m_2n_5}	\geq 
			{|x|^{n_4}|v|^{n_5}}
		}
		2^{m_2(2s+a)} \label{I12}\\
		&\quad=:I^1+I^2.
		\end{split}
	\end{equation}
	Regarding the first sum on the RHS of \eqref{I12}, it is obvious that
	\begin{equation*}
		\begin{split}
		I^1
		&\lesssim \sum_{0\leq a\leq b_2} \sum_{m_1\geq A}
		\sum_{m_2=B}^{m_1-3}
		\frac{2^{m_2(a+2s-n_5)} 2^{-m_1n_4}}{|x|^{n_4+b_1}|v|^{2+b_2+n_5-a}}
		\lesssim \sum_{0\leq a\leq b_2} \sum_{m_1\geq A}
		\frac{2^{B(a+2s-n_5)} 2^{-m_1n_4}}{|x|^{n_4+b_1}|v|^{2+b_2+n_5-a}}
		\\
		&\lesssim \sum_{0\leq a\leq b_2}  \frac{1}{|x|^{n_4+b_1}|v|^{2+b_2+n_5-a}} \left(\frac{1}{|x|^{n_4}|v|^{n_5}}\right)^{\frac{2s+a-n_5}{n_5}}  2^{A\left(-\frac{n_4}{n_5}(a+2s)\right)}\\
		&\lesssim\frac{1}{|v|^{2+2s+b_2-\frac{n_4(b_2+2s)}{n_4+n_5}}} \frac{1}{|x|^{b_1+\frac{n_4(b_2+2s)}{n_4+n_5}}}\\
		&\sim\frac{1}{|v|^{2+2s+b_2-\varepsilon}} \frac{1}{|x|^{b_1+\varepsilon}},
		\end{split}
	\end{equation*}
	the claim follows by choosing $n_5>2s+b_2$ appropriately.\\
	As for the second sum on the RHS of \eqref{I12}, we simply write that
	\begin{equation*}
		\begin{split}
		I^2
		&\lesssim \sum_{0\leq a\leq b_2}
		\sum_{m_2\leq D}
		\sum_{m_1=m_2+3}^{E}
		\frac{2^{m_2(a+2s)}}{|x|^{b_1}|v|^{2+b_2-a}}
		\\
		&\lesssim \sum_{0\leq a\leq b_2}
		\frac{1}{|x|^{b_1}|v|^{2+b_2-a}}
		\sum_{m_2\leq D}
		(-m_1)2^{m_2(a+2s)}\\
		&\lesssim \sum_{0\leq a\leq b_2}
		\frac{1}{|x|^{b_1}|v|^{2+b_2-a}}
		2^{D(a+2s)}\\
		&\lesssim \frac{1}{|v|^{2+2s+b_2-\frac{n_4(b_2+2s)}{n_4+n_5}}} \frac{1}{|x|^{b_1+\frac{n_4(b_2+2s)}{n_4+n_5}}}\\
		&\lesssim \frac{1}{|v|^{2+2s+b_2-\varepsilon}} \frac{1}{|x|^{b_1+\varepsilon}}.
		\end{split}
	\end{equation*}
	Therefore, we can arrive at 
	\begin{equation*}
		\sum_{m_2\leq m_1-3} \left|G_{m_1,m_2}\right| \lesssim \frac{1}{|v|^{2+2s+b_2-\varepsilon}} \frac{1}{|x|^{b_1+\varepsilon}}.
	\end{equation*}
	
	Next, our attention turns to $\frac{1}{2}<s<1$. Given its similarity to case $2.1$, we just summarize key findings here.\\
	\textbf{case $2.2: \frac{1}{2}<s<1 $}. We apply the integration by parts procedure three times with respect to $\eta$, resulting in
	\begin{equation*}
		K(1,x,v)=\frac{1}{(-iv)^3} \iint \partial_\eta^3 e^{-M(\xi,\eta)} e^{-i\xi\cdot x-i\eta\cdot v}  d\xi d\eta,
	\end{equation*}	
	which implies that
	\begin{equation*}
		\partial_x^{b_1} \partial_v^{b_2} K(1,x,v)
		=\sum_{a=0}^{b_2} \frac{C_{a,b_2}}{v^{3+b_2-a}} \iint (-i\xi)^{b_1} (-i\eta)^a \partial_\eta^3 e^{-M(\xi,\eta)} e^{-i\xi\cdot x-i\eta\cdot v} d\xi d\eta
		=:\sum_{m_1,m_2} H_{m_1,m_2}.
	\end{equation*}
	Thus, we can acquire the following results through analogous computation
	\begin{equation*}
		\begin{split}
		&\sum_{|m_1-m_2|\leq 2} \left|H_{m_1,m_2}\right|
		\lesssim
		\sum_{0\leq a \leq b_2} \frac{1}{|v|^{2+2s+b_1+b_2-\frac{n_6}{n_6+n_7}(2s+b_1+a-1)}} \frac{1}{|x|^{\frac{n_6}{n_6+n_7}(2s+b_1+a-1)}}
		\sim\frac{1}{|v|^{2+2s+b_1+b_2-\varepsilon}} \frac{1}{|x|^\varepsilon},\\
		&\sum_{m_2\geq m_1+3} \left|H_{m_1,m_2}\right|
		\lesssim  \sum_{0\leq a\leq b_2} \frac{1}{|v|^{3+b_2-a}} \left(\frac{1}{|v|^{n_8}}\right)^{\frac{2s+a+b_1-1}{n_8}} 
		\sim\frac{1}{|v|^{2+2s+b_1+b_2}},\\
		&\sum_{m_2\leq m_1-3} \left|H_{m_1,m_2}\right|
		\lesssim \sum_{0\leq a\leq b_2}	\frac{1}{|v|^{2+2s+b_2-\frac{n_9(a+2s-1)}{n_9+n_{10}}}} \frac{1}{|x|^{b_1+\frac{n_9(a+2s-1)}{n_9+n_{10}}}}
		\sim\frac{1}{|v|^{2+2s+b_2-\varepsilon}} \frac{1}{|x|^{b_1+\varepsilon}},
		\end{split}
	\end{equation*}
	the claim follows by choosing $n_6+n_7>2s+b_1+b_2-1,n_8>2s+b_1+b_2-1,  n_{10}>2s+b_2-1$ appropriately.
	Therefore, we can conclude that
	\begin{equation*}
		\sum_{m_2\leq m_1-3} \left|H_{m_1,m_2}\right|
		\lesssim \frac{1}{|v|^{2+2s+b_2-\varepsilon}} \frac{1}{|x|^{b_1+\varepsilon}}.
	\end{equation*}
	In conclusion, we infer that
	\begin{equation*}
		\sum_{m_1,m_2\in\mathbb{Z}} \widetilde{K}_{m_1,m_2}
		\lesssim \frac{1}{|v|^{2+2s+b_2-\varepsilon}}\frac{1}{|x|^{b_1+\varepsilon}}.
	\end{equation*}
	\begin{remark}
		It can be observed that, due to the range of $s$, in case $2$ it is necessary to perform a varying number of partial integrations.
	\end{remark}
	
	Finally, we briefly mention the case when $|x|\geq|v|$, where the situation seems to be delicate. Further details are omitted here. \\
	\textbf{case $3: |x|\geq|v|$}
	: Similar to case 2, we have
	\begin{equation*}
		\sum_{m_1,m_2\in\mathbb{Z}} \widetilde{K}_{m_1,m_2} \lesssim
		\frac{1}{|x|^{2+2s+b_1-\varepsilon}}\frac{1}{|v|^{b_2+\varepsilon}}.
	\end{equation*}	
	Therefore, combining the aforementioned three cases, for $\forall x,v \in \mathbb{R}$, the following inequality holds:
	\begin{equation*}
		\left|\partial_x^{b_1}\partial_v^{b_2}K(1,x,v)\right|\lesssim\frac{1}{\langle x,v \rangle ^{2+2s-2\varepsilon}\langle x\rangle^{\varepsilon+b_1}\langle v\rangle^{\varepsilon+b_2}},
	\end{equation*}
	which implies that
	\begin{equation*}
		\left|\partial_x^{b_1}\partial_v^{b_2}\mathcal{K}(1,x,v)\right|
		=\left|\partial_x^{b_1}\partial_v^{b_2}K(1,x,-x-v)\right|
		\lesssim\frac{1}{\langle x,x+v \rangle ^{2+2s-2\varepsilon}\langle x\rangle^{\varepsilon+b_1}\langle x+v\rangle^{\varepsilon+b_2}}.
	\end{equation*}
	This proves the Theorem \ref{estimate K}.
\end{proof}

\textbf{Acknowledgements:} The authors are grateful to Professor Quoc-Hung Nguyen, who introduced this project to us and patiently guided, supported, and encouraged us during this work.

\end{document}